\documentclass{article}

\pdfoutput=1

\usepackage{amsfonts,amssymb,amsmath,amsthm}

\usepackage{graphicx}
\usepackage{geometry}
\usepackage{xcolor}

\usepackage[colorlinks,
			pdffitwindow=true,
      			plainpages=false,
      			pdfpagelabels=true,
     			pdfpagemode=UseOutlines,
      			pdfpagelayout=SinglePage,
			bookmarks=false,
			colorlinks=true,
      			hyperfootnotes=false,
			linkcolor=blue,
			citecolor=green!50!black]{hyperref}

\usepackage{enumerate}

\newtheorem{definition}{Definition}
\newtheorem{theorem}[definition]{Theorem}
\newtheorem{lemma}[definition]{Lemma}
\newtheorem{corollary}[definition]{Corollary}

\newtheorem{remark}[definition]{Remark}
\newtheorem{proposition}[definition]{Proposition}
\newtheorem{example}[definition]{Example}

\newcommand{\R}{\mathbb{R}}
\newcommand{\C}{\mathbb{C}}
\newcommand{\N}{\mathbb{N}}

\newcommand{\X}{\mathcal{X}}
\newcommand{\U}{\mathcal{U}}
\newcommand{\ep}{\varepsilon}
\newcommand{\fnorm}[1]{\left\|#1\right\|_{\mathrm{F}}}

\def \bigo {\mathcal{O}}

\title{Optimal value functions for weakly coupled systems: a posteriori estimates}

\author{P\'eter Koltai and Oliver Junge\thanks{M3, Faculty for Mathematics, Technische Universit\"at M\"unchen, Boltzmannstr.\ 3, 85748 Garching. E-mails: koltai{@}ma.tum.de, junge@ma.tum.de}}

\date{September 17, 2012}

\begin{document}

%% maketitle must follow the abstract.
\maketitle                   % Produces the title.

\begin{abstract}
We consider weakly coupled LQ optimal control problems and derive estimates on the sensitivity of the optimal value function in dependence of the coupling strength. In order to improve these sensitivity estimates a ``coupling adapted'' norm is proposed. Our main result is that if a weak coupling suffices to destabilize the closed loop system with the optimal feedback of the uncoupled system then the value function might change drastically with the coupling.  As a consequence, it is not reasonable to expect that a weakly coupled system possesses a weakly coupled optimal value function. Also, for a known result on the connection of the separation operator and the stability radius a new and simpler proof is given.
\end{abstract}

% === INTRO ===
\section{Introduction}

When designing feedback controllers for high dimensional control systems using dynamic programming, the dimension of state space imposes natural limitations for existing methods:  In general, the representation of the optimal value function suffers from the curse of dimensionality.  An efficient approximation is only possible if the optimal value function possesses some regularity that we can exploit.  Our aim here is to investigate such regularity properties for systems which can be decomposed into a number of small subsystems which -- in a proper sense -- are only weakly coupled to each other.  This scenario is, e.g., motivated by so called \emph{networked control systems}.

More specifically, we consider a system consisting of several weakly coupled subsystems such that an optimal feedback of each subsystem, i.e.\ without the coupling to the other systems, is easily computable (for example, if it is an LQ system or if its state space dimension is $\leq 3$).  These feedbacks together form an optimal feedback for the uncoupled system, the \textit{naive feedback}.  Not surprisingly, the naive feedback may fail to stabilize the coupled system for particular coupling structures even if the coupling is weak.  In this paper, we aim to understand why and how the naive feedback fails, and how the optimal value function changes with increasing coupling strength.  

In particular, we consider the time-continuous linear quadratic regulator (LQR) problem, for which the optimal value function is a quadratic function of the form $V(x) = x^TPx$ with a symmetric positive definite matrix $P$.  In Section~\ref{sec:wclqr_cont} we recall some existing theory on the sensitivity analysis of the LQR problem. Since the norm in which the sensitivity is measured yields unsatisfactory estimates for the coupling of the optimal value function, in Section~\ref{sec:other_norm} the estimates are established in a coupling-adapted norm.  Similar estimates are derived in Section 5 for discrete time systems. Section~\ref{sec:disc_estim} contains a discussion about the obtained estimates.  Finally, we turn to the question on what these estimates tell us about the structure of the optimal value function if we can introduce a weak coupling such that the naive feedback fails to stabilize the coupled system. The answer is somewhat surprising: if a coupling (no matter how small) destabilizes the system controlled by the naive feedback, the optimal value function for the coupled system differs massively from the one of the uncoupled system; cf.~Section~\ref{sec:optimal value functiondestab}.

% === UNCOUPLED optimal value function ===
\section{The optimal value function, uncoupled and weakly coupled systems}	\label{sec:optimal value functionnoncoup}

We consider a control system
\begin{equation}\label{eq:cs}
\dot x = f(x,u),
\end{equation}
where the right hand side $f:\X\times\U\to \R^n$ is assumed to be $C^1$ and the state space $\X\subset\R^n$, $0\in\X$, and the set of controls $\U\subset\R^m$, $0\in\U$, are compact. The origin of the uncontrolled system is a (possibly unstable) equilibrium, i.e.\ we have that $f(0,0)=0$.  We are interested in constructing a feedback $u:\X\to \U$ such that the origin of the closed loop system $\dot x = f(x,u(x))$ is asymptotically stable.  To this end, we consider a continuous cost function $c:\X\times\U\to [0,\infty)$ which fulfills the condition $c(x,u)=0$ iff $x=0$.  Given an initial value $x_0\in\X$ and a (measurable) control function $u:[0,\infty)\to\X$, the cost along the associated trajectory $x(t;x_0,u)$ of (\ref{eq:cs}) is
\[
J(x_0,u)=\int_0^\infty c(x(t;x_0,u),u(t))\; dt,
\]
while the \emph{optimal value function} is
\[
V(x) = \inf\{J(x,u)\mid u:[0,\infty)\to\X \text{ measurable}\}.
\]
The optimal value function is the solution of the \emph{Hamilton-Jacobi-Bellman equation} (see~\cite{Sont98} Chapter~8)
\[
\inf_{u\in\U}\left\{\nabla V(x)\cdot f(x,u) + c(x,u)\right\}=0,
\]
with the boundary condition $V(0)=0$.  From this equation, an optimal feedback can be constructed by choosing a minimizing $u$ to a given $x$ (assuming that it exists).
 
We call a control system (\ref{eq:cs}) with cost function $c$ \emph{uncoupled} if one can decompose the state space and the set of controls into at least two subsystems, $\X=\bigotimes_{i=1}^k\X_i$ and $\U=\bigotimes_{i=1}^k\U_i$, $k\ge 2$, such that $f$ and $c$ can be written as 
\[
f(x,u) = \left[\begin{array}{c}f_1(x_1,u_1)\\\vdots\\f_k(x_k,u_k)\end{array}\right]
\] 
and $c(x,u) = \sum_{i=1}^k c_i(x_i,u_i)$, where $f_i:\X_i\times \U_i\to\R^{n_i}$, $c_i:\X_i\times\U_i\to [0,\infty)$ and $n_1+\ldots+n_k=n$.
Of course, for an uncoupled system, the optimal value function is simply the sum of the optimal value functions of the subsystems:
\[
V(x) = \sum_{i=1}^k V_i(x_i).
\]
Evidently, if $V$ is $C^2$, then $\partial_{i,j} V\equiv 0$, but also the converse is true (for the proof, see the appendix):
\begin{proposition}\label{prop:optimal value function_uncoupled}
For $V\in C^2(\X)$ the following statements are equivalent.
\begin{enumerate}[a)]
	\item There are $V_i\in C^2(\X_i)$, $i=1,\ldots,k$, such that $V(x) = \sum_{i=1}^k V_i(x_i)$.
	\item It holds that $\partial_{i,j} V\equiv 0$ for all $i,j\in\{1,\ldots,k\}$, $i\neq j$.
\end{enumerate}
\end{proposition}

Accordingly, a control system  $f$ is called \emph{weakly coupled} if again there is a  decomposition of $\X$ and $\U$ as above such that for all $i,j\in\{1,\ldots,k\}$, $i\neq j$, 
\[
\| \partial_i f_i(x,u) \| \gg \| \partial_j f_i(x,u) \|
\]
uniformly in $x$ and $u$, for a given norm $\|\cdot\|$ (on the $\X_i$).  We  define the \emph{coupling constant} as
\[
\ep_f := \max_{i\neq j} \sup_{x,u} \frac{\|\partial_j f_i(x,u) \|}{\| \partial_i f_i(x,u) \|}.
\]
Correspondingly, a system is weakly coupled if $\ep_f \ll 1$. Note that this notion is rather vague --- in the sequel we will make statements about the asymptotic case $\ep_f\to 0$.

In light of Proposition~\ref{prop:optimal value function_uncoupled} a natural question is whether for weakly coupled systems the mixed derivatives $\partial_{i,j}V$ for the optimal value function are uniformly small.

% === WEAKLY COUPLED LQR (CONTINUOUS TIME) ===
\section{Weakly coupled LQ systems with linear feedback --- continuous time}	\label{sec:wclqr_cont}

We now consider more specifically linear time invariant systems with quadratic cost function. A weakly coupled linear system is defined by
\begin{equation}\label{eq:LQ}
\dot x = Ax+Bu,
\end{equation}
where $A$ is a weakly coupled matrix, i.e.\ the diagonal blocks dominate the other ones (an implication of the definition above), and $B$ is a block diagonal matrix (i.e.\ the subsystems are controlled separately). The accumulated cost along a trajectory $x(t)=x(t;x_0,u(\cdot))$ is given by
\[
J(x_0,u(\cdot)) = \int_0^{\infty} x(t)^TQx(t) + u(t)^TRu(t)\, dt,
\]
where $Q$ and $R$ are symmetric positive definite (spd) matrices, both block diagonal to satisfy $c(x,u)=\sum_{i=1}^kc_i(x_i,u_i)$. The optimal value function is then the infimum over all possible control functions. Imposing no restrictions on $u$, the unique optimal control is realized by the feedback
\[
u(x) = -R^{-1}B^TP x,
\]
where $P$ is the unique spd solution of the Riccati equation
\begin{equation}
P BR^{-1}B^TP-P A-A^TP-Q = 0,
\label{eq:ric}
\end{equation}
and the optimal value function is given by $V(x)=x^TP x$; see~\cite{Sont98}, chapter 8.4.

Hence, the question whether the $\partial_{i,j} V$ are small, reduces to the question whether the  off-diagonal blocks of $P$ are small compared to the diagonal ones. Our answer will be of asymptotic nature, i.e.\ by considering small perturbations of an initially block diagonal (uncoupled) system matrix $A$.
Note, that for a block diagonal system matrix and an analogously block diagonal ansatz $P$ equation~\eqref{eq:ric} simplifies to $k$ uncoupled Riccati equations. Then, by uniqueness of the solution is $P$ block diagonal as well.

In order to analyze the effect of perturbations on $A$, we define the separation of two matrices; see e.g.~\cite{Vara79}.
\begin{definition}[Separation of matrices]
The separation of two square matrices $X$ and $Y$, $\mathrm{sep}(X,Y)$, is defined as the smallest singular value of $I\otimes X-Y^T\otimes I$, where $I$ is the identity and $\otimes$ denotes the Kronecker product.\footnote{It holds that $\mathrm{sep}(X,Y)=\mathrm{sep}(Y,X)$, and $\mathrm{sep}(X^T,Y^T)=\mathrm{sep}(X,Y)$.}
\end{definition}
The following estimate plays a central role in our considerations. Although more general results exist~\cite{KoPeCh86}, a proof of this one is also given, because we use the ideas in it again later.
\begin{theorem}	\label{thm:ric_cond}
Let $P(A)$ denote the solution of~\eqref{eq:ric} for fixed $B$, $Q$ and $R$, in dependence of $A$. If $P(A+\delta A) = P(A)+\delta P$, and $A_{cl}=A-BR^{-1}B^TP(A)$ denotes the closed-loop system matrix, it holds that
\begin{equation}
\fnorm{\delta P} \le \frac{2\fnorm{P(A)}}{\mathrm{sep}(A_{cl},-A_{cl}^T)}\fnorm{\delta A} + \bigo\left(\fnorm{\delta A}^2\right).
\label{eq:estim_lin}
\end{equation}
\end{theorem}
\begin{proof}
Define the function ${g(A,P)=PBR^{-1}B^TP-PA-A^TP-Q}$, where $P$ is assumed to be a symmetric matrix. By definition, ${g(A,P(A))=0}$ for all $A$. We have that $P(\cdot)$ is real-analytic in an appropriate neighborhood of $A$~\cite{Del83}. It also holds
\begin{eqnarray*}
\partial_A\, g(A,P)\cdot X & = & -PX-X^TP \\
\partial_P\, g(A,P)\cdot X & = & XBR^{-1}B^TP+PBR^{-1}B^TX-XA-A^TX.
\end{eqnarray*}
From the implicit function theorem we have
\[
DP(A) = -\partial_P\, g\left(A,P\left(A\right)\right)^{-1}\partial_A\, g(A,P(A)),
\]
if $\partial_P\, g\left(A,P\left(A\right)\right)$ is invertible. Following~\cite{Vara79}, one can see that this is the case if and only if $\mathrm{sep}(A_{cl},-A_{cl}^T)\neq0$ (in fact, $\partial_Pg(A,P(A))$ is a linear operator with matrix representation $-I\otimes A_{cl}-A_{cl}\otimes I$ if $X$ is represented as a long vector). Moreover, it holds that
\begin{equation}
\fnorm{DP(A)} \le \frac{2\fnorm{P(A)}}{\mathrm{sep}(A_{cl},-A_{cl}^T)}.
\label{eq:estimF}
\end{equation}
A Taylor expansion of $P$ in $A$ yields the claim.
\end{proof}

An immediate consequence of Theorem~\ref{thm:ric_cond} for our situation is that if we are given a weakly coupled linear system (\ref{eq:LQ}) and we can decompose $A=A_0+\delta A$ such that ${\fnorm{\delta A}\ll \fnorm{A_0}}$, and $\fnorm{DP(A_0)}$ is small, then the optimal value function is ``weakly coupled'', i.e.\ its mixed second derivatives are uniformly small.
\begin{remark}\quad
\begin{enumerate}[a)]
	\item Note that~\eqref{eq:estimF} is an a posteriori estimate, i.e.\ it does not characterize the sensitivity directly in terms of $A$, $B$, $Q$ and $R$, but involves the solution $P$ of the (uncoupled) Riccati equation.
	\item In~\cite{KoPeCh86} (see also~\cite{KGMP03}) the authors present a non-local perturbation analysis of the Riccati equation~\eqref{eq:ric}. This could be used to yield better bounds for the estimates of the mixed second derivatives of the optimal value function. With the present result we have no control about the quality of our estimate, since we have no information about the magnitude of higher order terms.\\
Their result states that $\fnorm{\delta P} \le g\left(\fnorm{\delta A}\right)$ for $\fnorm{\delta A}\in[0,a^*)$, with
\[
g(a) = \frac{1}{2d}\left(s-2a- \left(\left(s-2a\right)^2-8pda\right)^{1/2}\right),
\]
where $s = \mathrm{sep}\left(A_{cl},-A_{cl}^T\right)$, $d=\fnorm{BR^{-1}B^T}$ and $p=\fnorm{P(A)}$. Here, the right boundary point $a^*$ is defined as the (smaller) positive root of ${2a+2\sqrt{2dpa}-s=0}$.
\end{enumerate}
\end{remark}

% === OTHER NORM ===
\section{Coupling-adapted estimates}	\label{sec:other_norm}

\newcommand{\cnorm}[1]{\left\|#1\right\|_{C}}

A somewhat unsatisfactory property of the Frobenius norm $\fnorm{\cdot}$ is that $\fnorm{\delta A}$ is not closely related to the coupling constant of the system. Much more adequate for measuring the coupling of a system is the following norm, which focuses on the subsystems. Each matrix $X\in\R^{n\times n}$, with $n=\sum_{i=1}^kn_i$, can be partitioned blockwise, according to the subsystems: $X = (X_{ij})_{i,j=1}^k$, where $X_{i,j}\in\R^{n_i\times n_j}$. Define
\[
\cnorm{X} = \max_{i,j} \fnorm{X_{i,j}}.
\]
Note that this norm is \textit{not} sub-multiplicative, but it holds that
\[
\cnorm{AB} = \max_{i,j}\Big\|\sum_{\ell=1}^k A_{i,\ell}B_{\ell,j}\Big\|_F \le k\cnorm{A}\cnorm{B}.
\]
Adapting the proof of Theorem \ref{thm:ric_cond} to an uncoupled system, we obtain the following sensitivity estimate:
\begin{lemma}	\label{lem:estimC}
Consider the situation of Theorem \ref{thm:ric_cond} and suppose that the system is uncoupled. Then
\begin{equation}
\fnorm{(DP(A)\cdot \delta A)_{i,j}} \le \frac{\fnorm{P(A)_{i,i}}+\fnorm{P(A)_{j,j}}}{\mathrm{sep}((A_{cl})_{j,j},-(A_{cl})_{i,i}^T)}\max\big\{\fnorm{\delta A_{i,j}}, \fnorm{\delta A_{j,i}}\big\}.
\label{eq:estimC}
\end{equation}
\end{lemma}
Note that this block-wise sensitivity estimate directly implies an analogous estimate in the norm~$\cnorm{\cdot}$.
\begin{proof}
Note, that the system is uncoupled, hence $A$, $P(A)$ and $A_{cl}$ are all block-diagonal. It holds that block-wise 
\[
(\partial_A\, g(A,P)\cdot\delta A)_{i,j} = -P_{i,i}\delta A_{i,j} - \delta A_{j,i}^TP_{j,j}.
\]
Recall that $\partial_P\, g(A,P)\cdot X = -XA_{cl}-A_{cl}^TX =: -S$ for symmetric matrices $X$. It holds that block-wise 
\[
X_{i,j}(A_{cl})_{j,j} +  (A_{cl})_{i,i}^TX_{i,j} = S_{i,j} = -P_{i,i}\delta A_{i,j} - \delta A_{j,i}^TP_{j,j}.
\]
Solving this for $X_{i,j}$ implies the claim analogously to the proof of Theorem~\ref{thm:ric_cond}.
\end{proof}

The previous result also gives us a possibility to measure the effect on the coupling constant of the closed-loop system.
%\footnote{Here we consider the optimally controlled system as a system without control. It is still not investigated in what would our former definition of the coupling result in.}
\begin{corollary}
In the situation of Lemma~\ref{lem:estimC} the first order perturbation $\delta A_{cl}$ of the closed loop system matrix induced by the perturbation $\delta A$ of the system matrix $A$ satisfies
\begin{align}
\fnorm{\delta (A_{cl})_{i,j}} & \quad\dot{\leq}\quad 
\left(1+\fnorm{B_{i,i}R_{i,i}^{-1}B_{i,i}^T}\frac{\fnorm{P(A)_{i,i}}+\fnorm{P(A)_{j,j}}}{\mathrm{sep}((A_{cl})_{j,j},-(A_{cl})_{i,i}^T)}\right) \max\big\{\fnorm{\delta A_{i,j}},\fnorm{\delta A_{j,i}}\big\}
\end{align}
(where $\quad\dot{\leq}$ means that this inequality is satisfied up to higher order terms).
\end{corollary}
\begin{proof}
Use $A_{cl} = A-BR^{-1}B^TP(A)$.
\end{proof}

% === WEAKLY COUPLED LQR (DISCRETE TIME) ===
\section{Weakly coupled LQ systems with linear feedback --- discrete time}	\label{sec:wclqr_discr}

The considerations here are analogous to the ones in sections~\ref{sec:wclqr_cont} and \ref{sec:other_norm}. The discrete time control system is given by
\[
x_{m+1} = Ax_m+Bu_m, \quad m=0,1,2,\ldots.
\]
We assume $A$ to be weakly coupled and $B$ to be a block diagonal matrix. The cost accumulated along a trajectory $(x_m)_{m\in\N}=(x_m(x_0;(u_k)_{k\in\N}))_{m\in\N}$ is given by
\[
J(x_0,(u_k)_{k\in\N}) =\sum_{m=0}x_m^TQx_m + u_m^T R u_m,
\]
where $Q$ and $R$ are block diagonal spd matrices. If there are no restrictions on $u$ the optimal feedback is given by
\[
 u(x) = -\left(R+B^TP B\right)^{-1}B^TP A x,
\]
and the optimal value function is $V(x)=x^TP x$, with $P$ being the unique spd solution of the discrete Riccati equation (see \cite{Sont98}, chapter~8.4)
\begin{equation}
P = A^T\left(P-P B\left(R+B^TP B\right)^{-1}B^TP\right)A+Q.
\label{eq:DARE}
\end{equation}
In the following we derive a first order perturbation result on $P$ depending on $A$.

\newcommand{\sepd}{\mathrm{sep}^{\#}}
For arbitrary square matrices $M_1$, $M_2$ and $S$ the equation
\[
X-M_1^TXM_2 = S
\]
is solvable if the matrix $I-M_1^T\otimes M_2^T$ is nonsingular, and the absolute condition (w.r.t.\ the Frobenius norm) of the solution is given by $1/\sepd(M)$ with
\[
\sepd(M_1,M_2):= \sigma_{\min}(I-M_1^T\otimes M_2^T),
\]
where $\sigma_{\min}(\cdot)$ denotes the smallest singular value.

We show directly the block-wise sensitivity estimate.
\begin{theorem}
Let the linear discrete time system be uncoupled, and let $P(A)$ denote the solution of~\eqref{eq:DARE}, where $B$, $Q$ and $R$ are fixed; and let $A_{cl}$ denote the closed-loop matrix, i.e.\ ${A_{cl} = A-B(R+B^TPB)^{-1}B^TPA}$ . Then for some perturbation $\delta A$ of $A$ we have
\begin{align}
	\fnorm{(DP(A)\cdot\delta A)_{i,j}} &\le 
 \frac{\fnorm{P_{i,i} (A_{cl})_{i,i}}+\fnorm{(A_{cl})_{j,j}^T P_{j,j}}}{\sepd\big(\left(A_{cl}\right)_{i,i},\left(A_{cl}\right)_{j,j}\big)}\max\{\fnorm{\delta A_{i,j}},\fnorm{\delta A_{j,i}}\}.
\label{eq:estimCdiscr}
\end{align}
\end{theorem}
\begin{proof}
Just as in the proof of Theorem~\ref{thm:ric_cond}, we use the implicit function theorem. For this, let
\[
g^{\#}(A,P) = A^T\left(P B\left(R+B^TP B\right)^{-1}B^TP - P\right)A + P - Q.
\]
Setting ${\bar A = A-B(R+B^TPB)^{-1}B^TPA}$, we obtain
\begin{eqnarray*}
\partial_A g^{\#}(A,P)\cdot X & = & - X^TP\bar{A} - \bar{A}^TPX \\
\partial_P g^{\#}(A,P)\cdot X & = & X-\bar{A}^TX\bar{A}.
\end{eqnarray*}
Since all matrices involved are block diagonal, we have in particular the block wise equations
\begin{eqnarray*}
\left(\partial_A g^{\#}(A,P)\cdot X\right)_{i,j} & = & - X_{j,i}^TP_{j,j}\bar{A}_{j,j} - \bar{A}_{i,i}^TP_{i,i}X_{i,j} \\
\left(\partial_P g^{\#}(A,P)\cdot X\right)_{i,j} & = & X_{i,j}-\bar{A}_{i,i}^TX_{i,j}\bar{A}_{j,j}.
\end{eqnarray*}
Since $\bar A=A_{cl}$ for $P=P(A)$, considerations as in the proof of Theorem~\ref{thm:ric_cond} and Lemma~\ref{lem:estimC} imply the claim.
\end{proof}

\begin{remark}
For the discrete time case as well, in~\cite{KoPeCh86} the authors present a non-local perturbation analysis in the Frobenius norm for the solution of~\eqref{eq:DARE}.
\end{remark}

% === DISCUSSION OF THE ESTIMATES ===
\section{Coupling estimates in dependence on the number of subsystems} 	\label{sec:disc_estim}

In this section we compare the sensitivity estimates obtained in the Frobenius and coupling-adapted norms. The exposition is restricted to the continuous-time case, but analogous results hold in discrete time as well.

The estimates~\eqref{eq:estimF} and~\eqref{eq:estimC} are easier to compare if we see, that the denominators are the same:
\begin{proposition}
For a block-diagonal matrix we have 
\[
\mathrm{sep}(A,-A^T) = \min_{i,j} \mathrm{sep}(A_{i,i},-A_{j,j}^T).
\]
\end{proposition}
\begin{proof}
For block-diagonal $A$ the Lyapunov equation $A^TX+XA = Q$ is decoupled,
\[
A_{i,i}^TX_{i,j}+X_{i,j}A_{j,j}=Q_{i,j}.
\]
This means that the solution operator of the Lyapunov equation is block-wise decoupled, hence its singular values are the singular values of the blocks. This proves the claim.
\end{proof}

Consider now the coupling of $k$ identical subsystems, i.e.\ $A=I_{k\times k}\otimes A_0$, $B=I_{k\times k}\otimes B_0$, etc. Then $P = I_{k\times k}\otimes P_0$, where $P_0$ solves the Riccati equation with $A_0$, $B_0$, etc. In particular, the denominator in~\eqref{eq:estimF} stays constant in the number of subsystems $k$. Hence, the right hand side of the Frobenius norm estimate increases as $k^{1/2}$. Using~\eqref{eq:estimF} to obtain an estimate in $\cnorm{\cdot}$, we can use norm equivalence. Having $k$ subsystems, one verifies easily that
\[
\cnorm{A} \le \fnorm{A} \le k\cnorm{A}.
\]
Thus, we have
\[
\cnorm{\delta P} \le \fnorm{\delta P} \le \fnorm{DP(A)}\fnorm{\delta A} \le k\fnorm{DP(A)}\cnorm{\delta A}.
\]
In general, this would imply an estimate $\sim k^{3/2}$. 

However, the coupling-adapted norm estimate~\eqref{eq:estimC} shows that the first order perturbations do not depend on the number of subsystems \textit{at all}.  An important conclusion is that the coupling strength of the optimal value function does not increase with the number of subsystems.

\begin{remark}
This comparison gets important if we would try to devise an approximation technique for the OVF. The estimates suggest, that we should search for an approximation with error bounded by some analog of the the coupling-adapted norm for functions, because it does not scale with the number of subsystems. Based on our former considerations, in particular Proposition~\ref{prop:optimal value function_uncoupled}, an ideal candidate for the error indicator would be $\max_{i\neq j}\|\partial_{i,j}V\|_{\infty}$.
\end{remark}

% === optimal value function OF DESTABILIZED SYS ===
\section{The optimal value function of destabilized systems}	\label{sec:optimal value functiondestab}

% ---
\subsection{Continuous time}

In this section we investigate the structure of the optimal value function if a coupling destabilizes the optimally controlled (linear time-continuous) system. 

The norm of the smallest perturbation $\delta A$ that destabilizes $A$, the so-called \textit{stability radius} of $A$~\cite{HiPr86}, is given by
\[
r(A) := \min_{\delta A\in\C^{n\times n}}\left\{\|\delta A\|_2\,\big\vert\, \exists z\in\sigma(A+\delta A), \Re z=0\right\}.
\]

So if $\delta A$ destabilizes $A$, then $\|\delta A\|_2\ge r(A)$.
%; i.e.\ weakly coupled systems may turn unstable only if $r_{\R}(A)_{cl}$ is  small.
What can be said about the estimate~\eqref{eq:estim_lin} in these cases?
\begin{proposition}	\label{prop:sep_bound_stab_radius}
It holds that
\begin{equation}
\mathop{\mathrm{sep}}(A,-A^T) \le 2 r(A).
\label{eq:estim_sep}
\end{equation}
\end{proposition}
This result seems to be first stated in~\cite{KeHe90} Lemma 2.7; see also the remark after Corollary~2.5 in the same paper. They refer to Theorem~2.2 in~\cite{HeKe88} for the proof. If the stability radius is measured by the Frobenius norm instead of the spectral norm, the same statement can already be found in~\cite{VanLoa85}. His result follows from the aforementioned ones by simple norm estimates. Since our short proof differs significantly from the ones used in these papers, we include it in the appendix.

Proposition~\ref{prop:sep_bound_stab_radius} together with~\eqref{eq:estim_lin}, shows that if a weak coupling (i.e.\ with a small constant $\ep >0$) destabilizes a system, we have to take a strong coupling (i.e.\ with constant of size $\mathcal{O}(1)$) of the optimal value function into account. To see this, note that for a destabilizing coupling it holds that $\|\delta A\|_2\ge r(A_{cl})$, hence we have in leading order
\begin{equation}
\frac{\|\delta P\|_F}{\|P\|_F} \le \underbrace{\frac{2\|\delta A\|_F}{\mathrm{sep}(A_{cl},-A_{cl}^T)}}_{\ge 1}.
\label{eq:strong_sensitivity_estim}
\end{equation}

But does the coupling of the OVF stay small if $r(A_{cl})$ is not small? An answer can be given using the estimate of He~\cite{He97}, Theorem~4;
\begin{equation}
\frac{\pi r(A_{cl})^2\|A_{cl}\|_2}{2\|A_{cl}\|_2^2 + \pi r(A_{cl})^2} \le \mathrm{sep}(A_{cl},-A_{cl}^T).
\label{eq:HeEstimate}
\end{equation}
Assume that $r(A_{cl}) \ge \varrho\|A_{cl}\|_2$, with some robustness constant $\varrho\in (0,1)$. Combining the estimates~\eqref{eq:estim_lin} and~\eqref{eq:HeEstimate} yields in leading order
\[
\frac{\|\delta P\|_F}{\|P\|_F} \le \frac{2\left(\frac{2}{\varrho^2}+\pi\right)\|\delta A\|_F}{\pi\|A_{cl}\|_2}.
\]
Thus, if the robustness constant $\varrho$ is of order one, we expect a small coupling to induce also only a small coupling in the OVF (assuming that ``small coupling'' means $\|\delta A\|_F \ll \|A_{cl}\|_2$ here). 
\begin{remark}
The previous bound suggests, that in case of $\varrho\ll 1$ it might be possible that a perturbation $\|\delta A\|_F\ll r(A_{cl})$ induces large changes in the OVF, since $\|\delta A\|_F/\varrho^2$ can be big. Unfortunately, we did not find any examples supporting this claim.
\end{remark}
The above bounds did not consider whether the perturbation of the system matrix comes from a coupling or not. We would like to consider this topic in the following.
\begin{corollary}
For a block diagonal matrix $A$ we have
\[
\mathrm{sep}(A,-A^T) = \min_{i,j}\mathrm{sep}(A_{i,i},-A_{j,j}^T) \le 2\min_{i} r(A_{ii}).
\]
\end{corollary}
This could be interpreted as ``a system with robust subsystems yields a robust OVF''; however a formal proof would involve an estimate analogous to~\eqref{eq:HeEstimate}. We do not pursue this question here. Instead, we would like to characterize the sensitivity of the OVF in terms of the stability radius with respect to the coupling adapted norm~$\|\cdot\|_C$. Define
\[
r^C(A) := \min_{\delta A\in\C^{n\times n}}\left\{\cnorm{\delta A}\,\big\vert\, \exists z\in\sigma(A+\delta A), \Re z=0\right\},
\]
and note $\|A\|_2 \le k\cnorm{A}$. It follows $\text{sep}(A,-A^T)\le 2kr^C(A)$. If we restrict the class of allowed perturbations to the ones with zero diagonal blocks (i.e.\ ``coupling induced perturbations''), we immediately get the slight improvement $\text{sep}(A,-A^T)\le 2(k-1)r^C(A)$. With this bound, the analogous estimate to~\eqref{eq:strong_sensitivity_estim} would be
\begin{equation}
\frac{\|\delta P\|_C}{\|P\|_C} \le \underbrace{\frac{2\|\delta A\|_C}{\mathrm{sep}(A_{cl},-A_{cl}^T)}}_{\ge \frac{1}{k-1}},
\label{eq:strong_sensitivity_estim_coupl}
\end{equation}
which does not suggest the intuition of a strongly changing OVF under a weak destabilizing coupling provided $k$, the number of subsystems, is large. It would be desirable to find more sophisticated bounds of the separation in terms of the stability radius, in particular considering structured stability radii (adapted to perturbations induced by the coupling), in the sense of~\cite{KHP06}.

\begin{remark}
It should be noted that other characterizations of the OVF than~\eqref{eq:ric} exist, and hence other sensitivity estimates than~\eqref{eq:estimF} could be obtained; see~\cite{Meh91, KGMP03}.
\end{remark}

\paragraph{Examples.}

We consider two different types of examples. First, where the system matrix is strongly non-normal, so the pseudospectra~\cite{TrEm05} corresponding to small perturbations reach into the positive complex half plane. Second, where the control part of the cost (i.e.\ the matrix $R$) is much bigger than the state part of the cost (i.e.\ the matrix $Q$), hence the optimally controlled system is weakly stable in the sense that the spectral abscissa of the closed-loop system matrix is small (compared to the modulus of the corresponding eigenvalues).  In all examples we take $R$ to be the identity matrix. At the end of each example $\mathrm{sep}(A_{cl},-A_{cl}^T)$ and $r(A_{cl})$ are shown.

\begin{example}
The optimally controlled system has robust stable eigenvalues, i.e.\ the spectral abscissa is of order $\bigo(1)$. A weak coupling destabilizes the closed-loop system.
\[
A = \begin{pmatrix}
    1 & 1000 & 0 \\
    0 & -0.5 & 0 \\
    0 & 0 & -1
    \end{pmatrix}, \qquad
B = \begin{pmatrix}
		1 & 0\\
		0 & 0\\
		0 & 1
		\end{pmatrix}, \qquad
Q = I_{3\times 3}.
\]
Closed-loop eigenvalues: $-0.5,-\sqrt{2},-\sqrt{2}$. The coupling matrix
\[
\delta A = \begin{pmatrix}
		 0 & 0 & 0   \\
     0 & 0 & -0.1\\
     0.04 & 0 & 0
		 \end{pmatrix}
\]
moves the eigenvalues to $-1.6707 \pm 0.8163i, 0.0130$, and (with two digits precision)
\[
P = \begin{pmatrix}
		2.4 & 130 & 0 \\
		130 & 93000 & 0 \\
		0 & 0 & 0.41
\end{pmatrix}, \qquad
\delta P+P = \begin{pmatrix}
						 0.62 & 120 & -1.3 \\
						 120 & 33000 & -450 \\
						 -1.3 & -450 & 8.5
						 \end{pmatrix}
\]
$\mathrm{sep}(A_{cl},-A_{cl}^T) = 4.0\cdot 10^{-5}$. $r(A_{cl}) = 2.8\cdot 10^{-3}$.
\end{example}

\begin{example}
The optimally controlled system has robust stable eigenvalues. A weak coupling destabilizes the closed-loop system.
\[
A = \begin{pmatrix}
    1 & 50 & 0 & 0 & 0 \\
    0 & -0.5 & 50 & 0 & 0 \\
    0 & 0 & -0.5 & 50 & 0 \\
    0 & 0 & 0 & -0.5 & 0\\
    0 & 0 & 0 & 0 & 0.5
    \end{pmatrix}, \qquad
B = \begin{pmatrix}
		1 & 0\\
		0 & 0\\
		0 & 0\\
		0 & 0\\
		0 & 1
		\end{pmatrix}, \qquad
Q = I_{5\times 5}.
\]
The closed-loop eigenvalues are $-1.11, -1.41, -0.5, -0.5, -0.5$. The coupling matrix
\[
\delta A = \begin{pmatrix}
		 0 & 0 & 0 & 0 & 0 \\
		 0 & 0 & 0 & 0 & 0 \\
		 0 & 0 & 0 & 0 & 0 \\
		 0 & 0 & 0 & 0 & 0.1 \\
		 0.001 & 0 & 0 & 0 & 0
		 \end{pmatrix}
\]
moves the eigenvalues to $0.39 \pm 1.4i, -1.9 \pm 0.38i, -1.1$. $\mathrm{sep}(A_{cl},-A_{cl}^T) = 4.4\cdot 10^{-11}$. $r(A_{cl}) = 2.2\cdot 10^{-6}$.
\end{example}

\begin{example}
The optimally controlled system is only weakly stable due to an expensive control. A weak coupling destabilizes the closed-loop system
\[
A = \begin{pmatrix}
    10^{-2} & 1 & 0 & 0 \\
    -1 & 10^{-2} & 0 & 0 \\
    0 & 0 & 10^{-2} & 1 \\
    0 & 0 & - 1& 10^{-2}
    \end{pmatrix}, \qquad 
B = \begin{pmatrix}
		1 & 0\\
		0 & 0\\
		0 & 1\\
		0 & 0
		\end{pmatrix}, \qquad
Q = 10^{-4}I_{4\times 4}.
\]
The closed-loop eigenvalues are $-0.012 \pm 1i, -0.012 \pm 1i$. The coupling matrix
\[
\delta A = \begin{pmatrix}
		 0 & 0 & 5\cdot10^{-2} & 0 \\
		 0 & 0 & 0 & 0 \\
		 2\cdot10^{-2} & 0 & 0 & 0 \\
		 0 & 0 & 0 & 0  \\
		 \end{pmatrix}
\]
moves the eigenvalues to $-0.028 \pm 1i, 0.0036 \pm 1i$. The matrices of the associated optimal value functions are
\[
P = 10^{-2}\cdot\begin{pmatrix}
		 4.4 & -0.05 & 0 & 0 \\
		 -0.05 & 4.4 & 0 & 0 \\
		 0 & 0 & 4.4 & -0.05 \\
		 0 & 0 & -0.05 & 4.4
 	  \end{pmatrix},\qquad
P+\delta P = 10^{-2}\cdot\begin{pmatrix}
		 7.6 & -0.08 & 4.6 & -0.05 \\
		 -0.08 & 7.6 & -0.04 & 4.6 \\
		 4.6 & -0.04 & 3.8 & -0.04 \\
		 -0.05 & 4.6 & -0.04 & 3.8 
		 \end{pmatrix}.
\]
Clearly, the off-diagonal blocks of $P+\delta P$ are not small in comparison to the diagonal blocks. $\mathrm{sep}(A_{cl},-A_{cl}^T) = 0.025$. $r(A_{cl}) = 0.092$.
\end{example}

% --- optimal value function OF DESTABILIZED SYS (DISCRETE TIME) ===
\subsection{Discrete time}

Just as before, we compute a bound for the $\sepd(\cdot,\cdot)$ term in~\eqref{eq:estimCdiscr}. Recall
\[
\sepd(M_1,M_2) = \min_{\|X\|_F=1}\|X-M_1^TXM_2\|_F.
\]
We would like to bound $\sepd(A,A)$ in terms of the stability radius of $A$, defined by
\[
r(A) := \min_{\delta A\in\C^{n\times n}} \left\{\|\delta A\|_2\,\big\vert\, \exists \lambda\in\C:\, |\lambda|=1,\,\lambda\in \sigma(A+\delta A)\right\}.
\]
We have that
\begin{proposition}	\label{prop:sepd_bound_stab_radius}
\[
\sepd(A,A) \le r(A)\left(2+r\left(A\right)\right).
\]
\end{proposition}

Just as in the continuous-time case, this result, together with~\eqref{eq:estimCdiscr}, shows that if a weak coupling (i.e.\ with a small constant $\ep >0$) destabilizes a system, we have to take a strong coupling (i.e.\ with constant of size $\mathcal{O}(1)$) of the optimal value function into account.

% ---
\subsection{A discrete-time non-linear example}

The following system describes the temperature evolution of a fluid in two coupled reactors. It is a simplified version of~\cite{Stoe11}, and the constants may vary as well. The system consists of two one-dimensional subsystems:
\begin{equation}
\begin{aligned}
     \dot x_1 &= \frac{1}{0.0261} \left(2.133\cdot10^{-6}k_2\left(x_2-x_1\right) + q_c(u_1)(285.65-x_1) + 8.535\cdot 10^{-4}\right) \\
     \dot x_2 &= \frac{1}{0.0207} \left(1.795\cdot 10^{-5}(293.15-x_2) + 177.6\cdot 10^{-6} k_1 (x_1-x_2)+7.622\cdot10^{-5}u_2\right)
\end{aligned}
\label{eq:thermofluid}
\end{equation}
where the parameters $k_1=0.1$ and $k_2=0.2$ represent the coupling strength, $u_1,u_2\in[0,1]$ are control parameters, and
\[
q_c(u_1) = 1\cdot10^{-6}\left\{
\begin{array}{ll}
	83\cdot\left(1 - e^{\frac{-(u_1-0.15)}{0.1}}\right) & \text{for }u_1\ge0.15 \\
	0 & \text{other wise}.
\end{array}\right.
\]
The target set, where the system should be controlled to, is not a point, but a set, namely ${[295.03,299.71]\times[304.40,308.15]}$; and the cost functions are
\[
c_1(x_1,u_1) = 0.01\cdot\left|x_1 - 297.37\right| + (u_1-\bar u_1)^2,\qquad c_2(x_2,u_2) = 0.01\cdot\left|x_2 - 306.28\right| + (u_2-\bar u_2)^2,
\]
where $\bar u_{1,2}$ are chosen such that they hold the respective subsystem in the middle of the target (if no coupling is present). The state space is~$\X = [285.65,323.15]\times[293.15,323.15]$, and the control space is~$\U = [0,1]\times [0,1]$.

The time-$T$-map of~\eqref{eq:thermofluid} with $T=200$ will be the discrete-time system under consideration. We applied the method developed in~\cite{GrJu08a} to compute the OVF in the non-coupled, and coupled cases. Figure~\ref{fig:full2dcomputations} shows the results of the computation. Despite the weak coupling (the coupling constant is $\approx 0.1$), the OVF changes significantly: the relative difference of the coupled OVF to the non-coupled is more than~100\% in the~$L^2$~norm, and more than~80\% in the maximum norm. This is not surprising, since the coupling is chosen so that it de-stabilizes the optimally controlled non-coupled system. However, it is interesting to note that the best approximation of the coupled OVF in the $L^2$~norm by functions of the form $V_1(x_1)+V_2(x_2)$ yields a relative error of only 16\%! While the OVF changes significantly due to the weak coupling, it's approximability by non-coupled functions seems to be kept.
\begin{figure}[htb]
	\centering
		\includegraphics[width=0.4\textwidth]{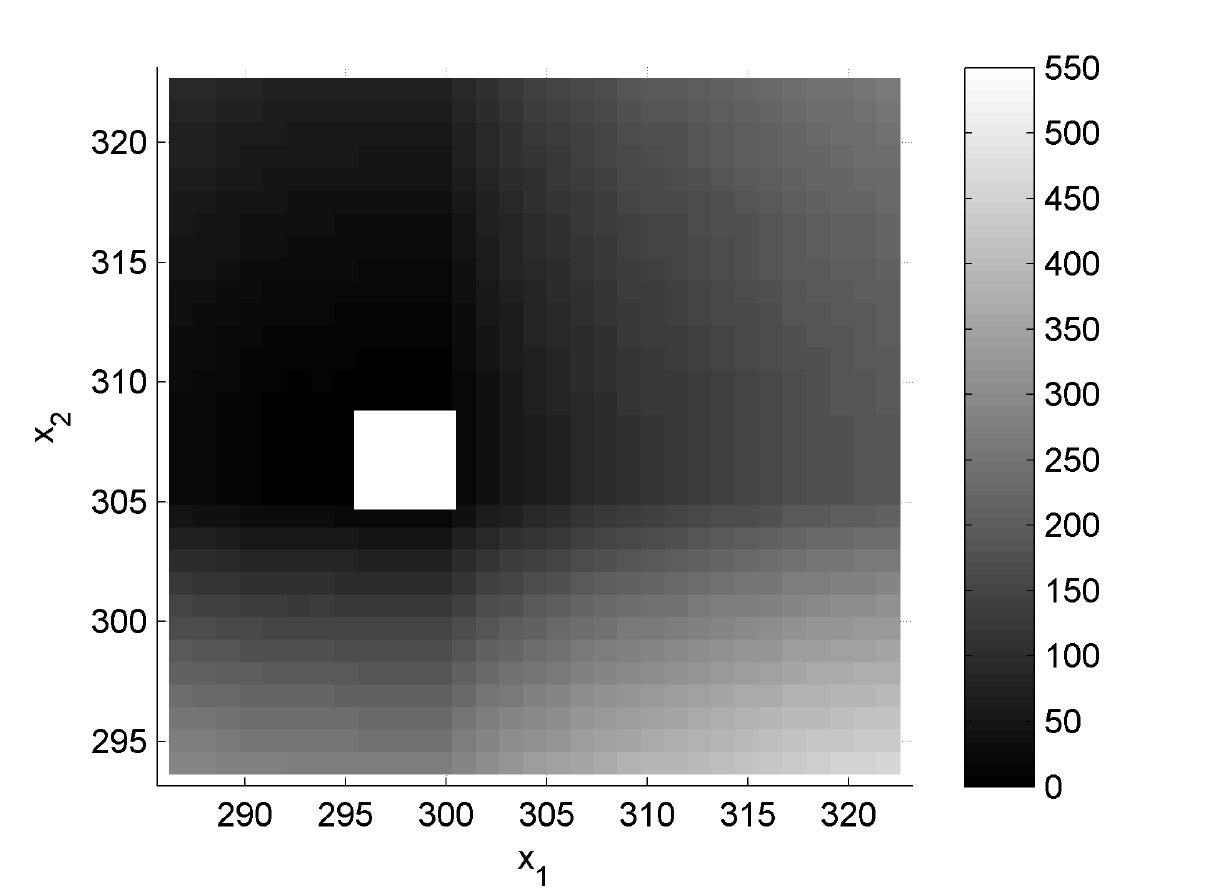} \quad
		\includegraphics[width=0.4\textwidth]{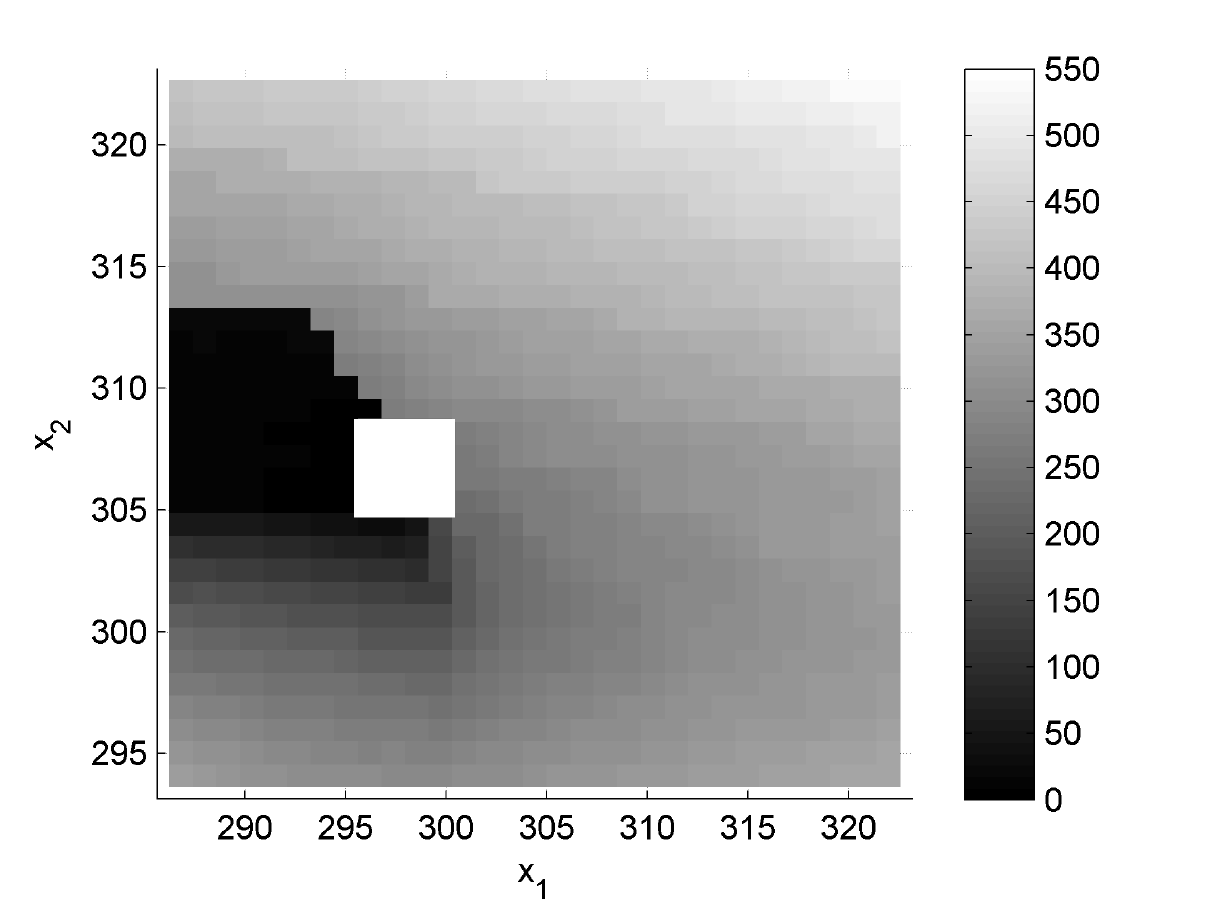}
	\caption{The OVF of the non-coupled (left) and the coupled (right) system. The white rectangle indicates the target set.}
	\label{fig:full2dcomputations}
\end{figure}

% === CONCLUSION ===
\section{Conclusions}

We have shown that even if one can divide an optimal control problem into subsystems which appear to be weakly coupled according to the structure of the Jacobian, the optimal value function still might not inherit this structure. We have also made an attempt to find a suitable norm for describing the sensitivity of the coupling strength of the optimal value function with respect to weak coupling in the system matrix. This ``coupling adapted'' norm turned out to be advantageous for upper bound estimates (cf.~Section~\ref{sec:disc_estim}), but does not seem to give a good intuition on the large sensitivity in the case of a large number of subsystems and a destabilizing coupling (cf.~Section~\ref{sec:optimal value functiondestab}). A better understanding is needed of how the coupling of the system matrix transfers into the coupling of the optimal value function. Certainly, more sophisticated ways of defining and detecting ``weak coupling'' are imaginable which would not show this undesired effect. Nevertheless, we have shown that the choice of coupling measure already has a significant effect on the obtained sensitivity bound.

On the way to designing (nearly) optimal controllers to weakly coupled systems a natural next question would be how does the optimal value function look like if additional robustness constrains (e.g.\ $r(A_{cl})$ has to be greater than a given threshold) have to be met by the closed loop system (if an optimal value function exists under such constraints at all); how it can be computed, and how does it compare to the optimal value function of the non-coupled system. Also, it might be possible to define cost functions such that the feedback for the associated uncoupled system still steers the weakly coupled system into some neighborhood of the target set.  These are questions for further investigations.

\section{Appendix}

\textbf{Proof of Proposition~\ref{prop:optimal value function_uncoupled}:}
\begin{proof}
``$\mathbf{\Rightarrow}$'': Trivial.\\
``$\mathbf{\Leftarrow}$'': In order to show the claim, we need following sublemma:\\
\textit{Sublemma}: If $V$ is $C^2$ in the $x_1$ and $x_2$ variables (let $z$ denote the other variables), and $\partial_{1,2}V(x_1,x_2,z)\equiv 0$, then $V(x_1,x_2,z) = V_1(x_1,z)+V_2(x_2,z)$, where $V_i$ is $C^2$ in the $x_i$ variable.\\
\textit{Proof of sublemma}: It holds that
\[
\partial_1V(x_1,x_2,z) -\partial_1 V(x_1,0,z) = \int_0^{x_2} \partial_{1,2}V(x_1,s,z)\,ds = 0.
\]
Again by the fundamental theorem of calculus we have
\[
V(x_1,x_2,z)-V(0,x_2,z) - V(x_1,0,z) + V(0,0,z) = 0,
\]
i.e.\ $V(x_1,x_2,z) = V_1(x_1,z) + V_2(x_2,z)$. From $V_1 = V-V_2$, and $\partial_1 V_2\equiv 0$, the rest of the claim follows immediately.

The proof of the proposition follows by induction. For $k=2$ the sublemma implies the statement. For the induction step, assume that \linebreak[4] ${V(x) = \sum_{i=1}^\ell V_i(x_i,z)}$, where $z=(x_{\ell+1},\ldots,x_k)$ and the $V_i$ are all $C^2$. It follows for all $i$
\[
0 = \partial_{i,\ell+1} V(x) = \partial_{i,\ell+1} V_i(x_i,x_{\ell+1},z'),
\]
where $z' = (x_{\ell+2},\ldots,x_k)$. The sublemma implies
\[
V_i(x_i,x_{\ell+1},z') = V_i^{(1)}(x_i,z') + V_i^{(2)}(x_{\ell+1},z'),
\]
$V_i^{(1)}$ and $V_i^{(2)}$ being $C^2$ in $x_i$ and $x_{\ell+1}$, respectively. Hence we have $V(x) = \sum_{i=1}^{\ell+1}W_i(x_i,z')$, where
\begin{eqnarray*}
W_i(x_i,z') & = & V_i^{(1)}(x_i,z'),\quad i=1,\ldots,\ell, \\
W_{\ell+1}(x_{\ell+1},z') & = & \sum_{j=1}^{\ell} V_j^{(2)}(x_{\ell+1},z').
\end{eqnarray*}
This completes the proof.
\end{proof}

\noindent \textbf{Proof of Proposition~\ref{prop:sep_bound_stab_radius}:}
\begin{proof}
We make use of the notion of \textit{pseudospectra}~\cite{TrEm05}. The set
\[
\sigma_{\ep}(A) := \left\{z \in\C\,\big\vert\,\exists \delta A\in\C^{n\times n}:\|\delta A\|_2\le\ep, z \in\sigma(A+\delta A)\right\}
\]
is the (complex) $\ep$-pseudospectrum, where $\sigma(\cdot)$ denotes the usual spectrum.

From Theorem~3.1~\cite{Vara79} we have
\begin{eqnarray*}
\textrm{sep}(A,-A^T) & \le & \textrm{sep}_{\lambda}(A,-A^T) := \min_{\ep_1,\ep_2}\left\{\ep_1+\ep_2\,\big\vert\,\sigma_{\ep_1}(A)\cap\sigma_{\ep_2}(-A^T)\neq\emptyset\right\} \\
									 & \le & \min_{\ep_1,\ep_2}\left\{\ep_1+\ep_2\,\big\vert\,\exists z\in i\R, z\in\sigma_{\ep_1}(A), z\in\sigma_{\ep_2}(-A^T)\right\} \\
									 & = & 2r(A).
\end{eqnarray*}
The inequality in the second line comes from restricting the minimization to the imaginary axis; the equation in the third line comes from the fact that $i\omega\in\sigma(A+\delta A)$ for some $\omega\in\R$ implies $i\omega\in\sigma(-A^T-\delta A^*)$, where $A^*$ denotes the conjugate transpose of~$A$; and that $\|A\|_2 = \|A^*\|_2$.
\end{proof}

\noindent \textbf{Proof of Proposition~\ref{prop:sepd_bound_stab_radius}:}
\begin{proof}
Assume we have $\lambda\in\C$, $|\lambda|=1$, non-negative numbers $\ep_{1,2}$, vectors $\|v_{1,2}\|_2=1$ and $\|w_{1,2}\|_2=\ep_{1,2}$ such that for the real matrices $M_{1,2}$ holds
\[
(M_i^T-\lambda I)v_i = w_i,\quad i=1,2.
\]
It follows ($x^*$ denotes the conjugate transpose of $x$)
\begin{eqnarray*}
w_1w_2^* & = & (M_1^T-\lambda I)v_1v_2^*(M_2-\bar\lambda  I) \\
				 & = & M_1^Tv_1v_2^*M_2 - \lambda v_1v_2^*M_2 - \bar\lambda M_1^Tv_1v_2^* + v_1v_2^* \\
				 & = & M_1^Tv_1v_2^*M_2 - \lambda v_1(\bar\lambda v_2^*+w_2^*) - \bar\lambda (\lambda v_1+w_1)v_2^* + v_1v_2^* \\
				 & = & M_1^Tv_1v_2^*M_2 - v_1v_2^* - \lambda v_1w_2^* - \bar\lambda w_1 v_2^*.
\end{eqnarray*}
With this equation and by setting $X = v_1v_2^*$ we conclude that $\sepd(M_1,M_2)\le\ep_1\ep_2+\ep_1+\ep_2$.

Since a $\lambda\in\C$ is in the (complex) $\ep$-pseudospectrum of $A$ if there is a vector $v$ of unit modulus such that ${\|(A-\lambda I)v\|_2=\ep}$ (see~\cite{TrEm05}), we can find a $\|v\|_2=1$ such that $\|(A-\lambda I)v\|_2=r(A)$.

Setting $v_1=v_2=v$ and $M_{1,2}=A$ in the above computation yields the claim.
\end{proof}

%\begin{acknowledgement}
%\end{acknowledgement}

\bibliographystyle{alpha}
\bibliography{References}

\end{document}